\documentclass[10pt]{amsart}

\usepackage{latexcad}
\usepackage{natbib}

\theoremstyle{definition}
\newtheorem{definition}{Definition}[section]

\theoremstyle{theorem}
\newtheorem{lemma}[definition]{Lemma}
\newtheorem{theorem}[definition]{Theorem}
\newtheorem{proposition}[definition]{Proposition}
\newtheorem{corollary}[definition]{Corollary}

\def\disp#1{\mathrm{d}(#1)}
\def\stre#1#2#3{\mathrm{s}^{#1}_{#2}(#3)}
\def\sstr#1{\mathrm{s}^*(#1)}

\title{How permutations displace points and stretch intervals}

\author{Daniel Daly}

\email{ddaly@du.edu, petr@math.du.edu}

\author{Petr Vojt\v{e}chovsk\'y}

\address{Department of Mathematics, University of Denver, 2360 S Gaylord St,
Denver, CO 80208, U.S.A.}

\begin{document}

\maketitle

\begin{abstract}
Let $S_n$ be the set of permutations on $\{1,\,\dots,\,n\}$ and $\pi\in S_n$.
Let $\disp{\pi}$ be the arithmetic average of $\{|i-\pi(i)|;\;1\le i\le n\}$.
Then $\disp{\pi}/n\in[0,\,1/2]$, the expected value of $\disp{\pi}/n$
approaches $1/3$ as $n$ approaches infinity, and $\disp{\pi}/n$ is close to
$1/3$ for most permutations. We describe all permutations $\pi$ with maximal
$\disp{\pi}$.

Let $\mathrm{s}^+(\pi)$ and $\mathrm{s}^*(\pi)$ be the arithmetic and geometric
averages of $\{|\pi(i)-\pi(i+1)|;\;1\le i<n\}$, and let $M^+$, $M^*$ be the
maxima of $\mathrm{s}^+$ and $\mathrm{s}^*$ over $S_n$, respectively. Then
$M^+=(2m^2-1)/(2m-1)$ when $n=2m$, $M^+ = (2m^2+2m-1)/(2m)$ when $n=2m+1$, $M^*
= (m^m(m+1)^{m-1})^{1/(n-1)}$ when $n=2m$, and, interestingly, $M^* =
(m^m(m+1)(m+2)^{m-1})^{1/(n-1)}$ when $n=2m+1>1$. We describe all permutations
$\pi$, $\sigma$ with maximal $\mathrm{s}^+(\pi)$ and $\mathrm{s}^*(\sigma)$.
\end{abstract}

\section{Motivation and introduction}

Allow us to begin with a motivation from the area of turbo coding \cite{HW,
Sklar}: Starting with the very first example \cite{BGT}, every turbo code
employs a permutation, called the \emph{interleaver}. Although the interleaver
has several functions within the coding process, its main objective is to
scramble the input bits so that input sequences with a few nonzero bits do not
produce output sequences with many nonzero bits, upon being encoded with a
convolutional code. The interleaver is typically of length at least one
thousand.

While it is easy to simulate the transmission channel and measure the
performance of a turbo code with a particular interleaver statistically, it
appears to be difficult to characterize those permutations that will perform
well as interleavers without actually testing them. Indeed, early
publications on turbo coding recommend to select the interleaver at
random---an advice still followed in practice.

Nevertheless, it has now become clear that it is sometimes possible to match
or outperform random interleavers with deterministic or semi-random
interleavers by carefully analyzing the channel and the decoding algorithm,
among other parameters.

As an illustration, we mention three properties of permutations that have been
suggested in the literature as desirable for the purposes of turbo coding. Let
$n$ be an integer, $S_n$ the set of permutations on $\{1,\dots,n\}$, and
$\pi\in S_n$. Then:
\begin{enumerate}
\item[(a)] $\pi$ should have no fixed points and, more generally, the
\emph{delay} $i-\pi(i)$ should be far from zero for every $i$ \cite{GMBC,
SSSN},

\item[(b)] the quantity $\min\{|i-j|+|\pi(i)-\pi(j)|;\;1\le i<j\le n\}$
should be large \cite{DD, SSSN},

\item[(c)] the \emph{dispersion} $|\{(i-j,\pi(i)-\pi(j));\;1\le i<j\le
n\}|\cdot (n(n-1)/2)^{-1}$ should be large \cite{TC, HW}.
\end{enumerate}
Viewed in this way, interleaver design is very much a combinatorial problem.

In this paper, we define and discuss two properties of permutations similar to
(a)--(c), namely \emph{displacement} and \emph{stretch}. Most of our arguments
are combinatorial in nature and no knowledge of coding is needed. While the
results obtained here can be considered complete from the mathematical point of
view (in their narrow scope), the investigation of the impact of the results on
turbo coding is in preliminary stages, is carried out by a different group of
researchers, and is mentioned only once below.

Here are the two properties and a summary of results:

\subsection{Displacement}\label{Ss:D}

For $\pi\in S_n$, let
\begin{equation}\label{Eq:Displacement}
    \disp{\pi} = \sum_{i=1}^n \frac{|i-\pi(i)|}{n}.
\end{equation}
The value $\disp{\pi}$ has been defined in \cite[Thm.\ 2]{GMBC}, where it is
called descriptively \emph{the average of the absolute values of the delays}.
We prefer to call it the \emph{displacement} of $\pi$, and $\disp{\pi}/n$ the
\emph{normalized displacement} of $\pi$.

We prove that the normalized displacement of a permutation ranges between $0$
and $1/2$, and we find all permutations with extreme displacement. Among all
permutations in $S_n$, the average normalized displacement approaches $1/3$
as $n$ approaches $\infty$. Moreover, the distribution of displacements is
such that a long, randomly chosen permutation will very likely have
normalized displacement close to $1/3$.

Hence, by selecting the interleaver at random, the class of permutations with
large or small displacement is rarely (never!) put to the test. Preliminary
results of Ramya Chandramohan \cite{Ramya} indicate that an S-random
interleaver (see \cite{DD}) with larger than average displacement performs
slightly better than an S-random interleaver.

It is easy to construct permutations with normalized displacement arbitrarily
close to a given $0\le d\le 1/2$. The problem is more difficult when the
permutation is supposed to have additional properties.

\subsection{Stretching}\label{Ss:S}

The two quantities defined in (b), (c) are telling us something about how the
permutation $\pi$ stretches intervals. To measure the average stretch of an
arbitrary collection $\mathcal A$ of subsets of $N=\{1,\dots,n\}$, we propose
the following two definitions:

For $A\subseteq N$, let $\mathrm{diam}(A)=\max\{i;\;i\in A\} - \min\{i;\;i\in
A\}$. When $\mathcal A\subseteq 2^N$ and $\pi\in S_n$, let
\begin{equation}\label{Eq:AStretch}
    \stre{+}{\mathcal A}{\pi} = |\mathcal A|^{-1}\cdot\left(\sum_{A\in\mathcal A}
    \frac{\mathrm{diam}(\pi(A))}{\mathrm{diam}(A)}\right),
\end{equation}
and
\begin{equation}\label{Eq:Stretch}
    \stre{*}{\mathcal A}{\pi} = \left(\prod_{A\in\mathcal
    A}\frac{\mathrm{diam}(\pi(A))}{\mathrm{diam}(A)}\right)^{1/|\mathcal A|}.
\end{equation}
We call both formulas the \emph{stretch of $\pi$ with respect to $\mathcal
A$}. Formula \eqref{Eq:Stretch}, which gives equal weight to relative
stretching and shrinking, is merely the multiplicative version of
\eqref{Eq:AStretch}.

Since the formulas \eqref{Eq:AStretch}, \eqref{Eq:Stretch} emphasize average
stretch instead of extreme stretch, they become trivial when $\mathcal A=2^N$,
$\mathcal A=\{\{i,j\};\; i<j\in N\}$, etc. However, they are not meaningless.
For instance, when $n=3$ and $\mathcal A=\{\{1,2\},\{2,3\}\}$, we have
$\stre{+}{\mathcal A}{(1,3,2)} = 3/2> 1=\stre{+}{\mathcal A}{\mathrm{id}}$ and
$\stre{*}{\mathcal A}{(1,3,2)} = \sqrt{2} > 1 = \stre{*}{\mathcal
A}{\mathrm{id}}$, as one would expect.

It appears to be hopelessly complicated to analyze $\mathrm{s}^+$ and
$\mathrm{s}^*$ for an arbitrary collection $\mathcal A$. We therefore focus on
stretching with respect to $\mathcal B=\{\{i,i+1\};\;1\le i<n\}$.

Roughly speaking, the additive formula \eqref{Eq:AStretch} with $\mathcal
A=\mathcal B$ is maximized by any permutation that starts in the middle of the
interval $N$ and keeps oscillating between the two halves of $N$. The
multiplicative formula \eqref{Eq:Stretch} with $\mathcal A=\mathcal B$ leads to
a much more intricate solution. The maximum of $\mathrm{s}^*$ is
\begin{align*}
    &(m^m m^{m-1})^{1/(n-1)},&\text{when $n=2m$, and}\\
    &(m^m (m+1)(m+2)^{m-1})^{1/(n-1)},&\text{when $n=2m+1$}.
\end{align*}
(See Acknowledgement.) Furthermore, the maximum is attained by two permutations
when $n$ is even, and by four permutations when $n>1$ is odd.

\section{Displacement}

% AVERAGE DISPLACEMENT

\subsection{Average displacement}

We are first going to determine the average value of $\disp{\pi}$ over all
permutations $\pi\in S_n$. The formula \eqref{Eq:Avg} can be obtained by
combining Theorems 2 and 4 of \cite{GMBC} but our proof is shorter and more
straightforward.

\begin{theorem}\label{Pr:Avg}
Let $n\ge 1$ be an integer. Then
\begin{equation}\label{Eq:Avg}
    \frac{1}{n!}\sum_{\pi\in S_n}\disp{\pi} = \frac{n^2-1}{3n}.
\end{equation}
\end{theorem}
\begin{proof}
Pick $m\in N$. Since the number of permutations $\pi\in S_n$ mapping $m$ onto
some $m'$ is equal to $(n-1)!$, we have
\begin{displaymath}
    \frac{1}{n!}\sum_{\pi\in S_n}|m-\pi(m)| = \frac{((m-1)+\cdots + 1) +
    (1 + \cdots + (n-m))}{n}.
\end{displaymath}
Thus
\begin{align*}
    \frac{1}{n!}\sum_{\pi\in S_n}\disp{\pi}&=
        \frac{1}{n!}\sum_{\pi\in S_n}\frac{1}{n}\sum_{m=1}^n|m-\pi(m)|
        =\frac{1}{n}\sum_{m=1}^n\frac{1}{n!}\sum_{\pi\in S_n}|m-\pi(m)|\\
    &=\frac{1}{n}\sum_{m=1}^n\frac{(m-1)m+(n-m)(n-m+1)}{2n}\\
    &=\frac{1}{n}\sum_{m=1}^n \frac{(n-m)^2+(m-1)^2+n-1}{2n}.
\end{align*}
We now note that
\begin{displaymath}
    \sum_{m=1}^n (n-m)^2=\frac{(n-1)n(2n-1)}{6}=\sum_{m=1}^n (m-1)^2,
\end{displaymath}
and the result follows.
\end{proof}

The average displacement over all permutations from $S_n$ is therefore about
$n/3$. Asymptotically:

\begin{corollary}\label{Cr:LimDp}
We have
\begin{displaymath}
    \lim_{n\to\infty}\frac{1}{n}\cdot\frac{1}{n!}\sum_{\pi\in S_n}
    \disp{\pi} = \frac{1}{3}.
\end{displaymath}
\end{corollary}

%% EXTREME DISPLACEMENT

\subsection{Extreme displacement}

\noindent The minimal displacement $\disp{\pi}=0$ is attained by exactly one
permutation---the identity permutation. The dual question concerning maximal
displacement is more interesting.

Let us call a permutation $\pi\in S_n$ \emph{crossing} if for every $i$, $j$ in
$N$ the two closed intervals $[i,\pi(i)]$, $[j,\pi(j)]$ intersect (possibly at
a single point). Otherwise, $\pi$ is said to be \emph{noncrossing}.

%%%%% (FIGURE) [Fg:Noncrossing] %%%%%%%%%%%%%%%%%%%%%%%%%%%%%%%%%%%%%%%%%%%%%%%%
\setlength{\unitlength}{1.1mm}
\begin{figure}[ht]
    \begin{small}
    \centering
    \input{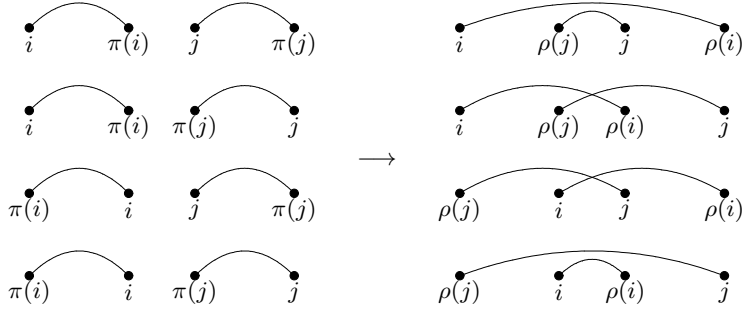}
    \end{small}
    \caption[]{Increasing displacement of noncrossing permutations.}
    \label{Fg:Noncrossing}
\end{figure}

\begin{lemma}\label{Lm:Increase}
Let $\pi\in S_n$ be a noncrossing permutation. Then there is $\rho\in S_n$
with $\disp{\rho}>\disp{\pi}$.
\end{lemma}
\begin{proof}
Since $\pi$ is noncrossing, there are $i<j$ in $N$ such that the intervals
$[i,\pi(i)]$, $[j,\pi(j)]$ are disjoint. Let $\rho = \pi\circ (i,j)$, where
the transposition $(i,j)$ is applied first. Then
\begin{displaymath}
    |i-\rho(i)|+|j-\rho(j)| = |i-\pi(i)|+|j-\pi(j)|
        +2(\min\{j,\pi(j)\}-\max\{i,\pi(i)\}),
\end{displaymath}
which is perhaps best apparent from Figure \ref{Fg:Noncrossing}. Since $i<j$
and $\pi$ is noncrossing, the term $\min\{j,\pi(j)\}-\max\{i,\pi(i)\}$ is
positive, proving that $\disp{\rho}>\disp{\pi}$.
\end{proof}

Now when we have seen that only crossing permutations can attain maximal
displacement, we characterize them.

\begin{lemma}\label{Lm:Crossing} Let $\pi\in S_n$.
If $n=2m$ then $\pi$ is crossing if and only if it maps $\{1$, $\dots$, $m\}$
onto $\{m+1$, $\dots$, $n\}$. If $n=2m+1$ then $\pi$ is crossing if and only
if it maps $\{1,\dots,m\}$ to $\{m+1,\dots,n\}$ and $\{m+2,\dots,n\}$ to
$\{1,\dots,m+1\}$.
\end{lemma}
\begin{proof}
Suppose first that $n=2m$. Assume that $\pi$ is crossing. If there is
$i\in\{1,\dots,m\}$ with $\pi(i)\in\{1,\dots,m\}$ then, by the pigeon-hole
principle, there must also be $j\in\{m+1,\dots,n\}$ with
$\pi(j)\in\{m+1,\dots,n\}$. But then the points $i$, $j$ and their images
$\pi(i)$, $\pi(j)$ witness that $\pi$ is noncrossing, a contradiction.
Conversely, every permutation $\pi$ mapping $\{1,\dots,m\}$ onto
$\{m+1,\dots,n\}$ must also map $\{m+1,\dots,n\}$ onto $\{1,\dots,m\}$, and
hence is a crossing permutation.

Now suppose that $n=2m+1$. Assume that $\pi$ is crossing and that
$\pi(m+1)\ge m+1$. Then the image of $\{1,\dots,m\}$ must be contained in
$\{m+1,\dots,n\}$, which forces $\pi$ to map $\{m+2,\dots,n\}$ onto
$\{1,\dots,m\}$. Similarly when $\pi$ is crossing and $\pi(m+1)\le m+1$.
Conversely, assume that $\pi$ maps $\{1,\dots,m\}$ to $\{m+1,\dots,n\}$ and
$\{m+2,\dots,n\}$ to $\{1,\dots,m+1\}$. Looking at two points at a time, it
is easy to see that $\pi$ is crossing.
\end{proof}

Note that the odd case of Lemma \ref{Lm:Crossing} imposes no restriction on the
image of the midpoint $m+1$. Nevertheless, once $m+1$ is mapped somewhere,
condition (ii) of Lemma \ref{Lm:Crossing} forces $\pi$ to behave in a certain
way. For instance, when $\pi(m+1)>m+1$, it follows that $\pi^{-1}(m+1)<m+1$. We
will need this fact in the next theorem.

\begin{theorem}\label{Th:Extreme}
Given $n\ge 1$, let $d_n = \max\{\disp{\pi};\;\pi\in S_n\}$, and $D_n =
\{\pi\in S_n;\;\disp{\pi}=d_n\}$. Then $\pi\in D_n$ if and only if $\pi$ is
crossing. Moreover, $d_n=n/2$ when $n$ is even, and $d_n=(n-1)(n+1)(2n)^{-1}$
when $n$ is odd.
\end{theorem}
\begin{proof}
Suppose that $n=2m$, and let $\pi\in S_n$ be a crossing permutation. By Lemma
\ref{Lm:Crossing}, $\pi$ maps $\{1,\dots,m\}$ onto $\{m+1,\dots,n\}$ and vice
versa. Therefore
\begin{eqnarray*}
    n\disp{\pi}&=&\sum_{i=1}^m|i-\pi(i)|+\sum_{i=m+1}^n|i-\pi(i)|\\
    &=&\sum_{i=1}^m(\pi(i)-i)+\sum_{i=m+1}^n(i-\pi(i))\\
    &=&2\left(\sum_{i=m+1}^n i -\sum_{i=1}^m i\right)
    =2\left(\frac{n(n+1)}{2}-2\cdot\frac{m(m+1)}{2}\right)
    =\frac{n^2}{2}.
\end{eqnarray*}
This short calculation proves that, as far as $\pi$ is crossing, the value of
$\disp{\pi}$ is independent of $\pi$ and is equal to $n/2$. The set $D_n$
then coincides with crossing permutations by Lemma \ref{Lm:Increase}, and
$d_n=n/2$ follows.

%%%%% (FIGURE) [Fg:Noncrossing] %%%%%%%%%%%%%%%%%%%%%%%%%%%%%%%%%%%%%%%%%%%%%%%%
%\setlength{\unitlength}{1.4mm}
%\begin{figure}[ht]
%    \centering
%    \input{fixed.lp}
%    \caption[]{Proof of Theorem \ref{Th:Extreme}.}
%    \label{Fg:Fixed}
%\end{figure}

Suppose that $n=2m+1$, and let $\pi\in S_n$ be a crossing permutation. If
$\pi(m+1)\ne m+1$, we construct a crossing permutation $\rho$ with
$\rho(m+1)=m+1$ satisfying $\disp{\rho}=\disp{\pi}$ as follows: Without loss of
generality, suppose $c=\pi(m+1)>m+1$. Then $a=\pi^{-1}(m+1)<m+1$, as we have
remarked before this theorem. Let $\rho(a)=c$, $\rho(c)=a$, $\rho(m+1)=m+1$ and
$\rho(k)=\pi(k)$ for $k\not\in\{a,m+1,c\}$.
%The situation is depicted in Figure \ref{Fg:Fixed}.
By the construction, $\disp{\pi}=\disp{\rho}$.

We can therefore assume that the crossing permutation $\pi$ fixes $m+1$.
Then, by Lemma \ref{Lm:Crossing},
\begin{eqnarray*}
    n\disp{\pi}&=&\sum_{i=1}^m(\pi(i)-i)+\sum_{i=m+2}^n(i-\pi(i))\\
    &=&2\left(\sum_{i=m+2}^n i - \sum_{i=1}^m i\right) =
    2m(m+1)=\frac{(n-1)(n+1)}{2}.
\end{eqnarray*}
As in the even case, we see that the value of $\disp{\pi}$ does not depend on
$\pi$, that $D_n$ consists exactly of all crossing permutations, and that
$d_n=(n-1)(n+1)(2n)^{-1}$.
\end{proof}

%% DISTRIBUTION OF DISPLACEMENTS

\subsection{Distribution of displacements}

\noindent The reader may wish to select a permutation $\pi$ of length $n=1000$
at random and calculate its displacement $\disp{\pi}$. We predict that
$330<\disp{\pi}<336$. We could be wrong, of course, as there are permutations
with displacement ranging from $0$ to $n/2$. Using the characterization of
permutations with maximal displacement (Lemma \ref{Lm:Crossing}), we count
exactly $(m!)^2$ such permutations in the even case $n=2m$. The ratio
$(2m)!/((m!)^2)$ approaches $0$ exponentially fast, so such permutations are
rare. This is an instance of a much more general notion known to measure
theorists as \emph{concentration of measure phenomena}. Let us talk about it
briefly, imitating \cite[Ch.\ 6]{MS}.

Let $(X,\rho,\mu)$ be a metric space equipped with a Borel probability measure
$\mu$. For a subset $A$ of $X$ and $\varepsilon>0$ define $A_\varepsilon=\{x\in
X;\;\rho(x,A)\le\varepsilon\}$, where $\rho(x,A)$ is the distance of $x$ from
the set $A$. The \emph{concentration function} $\alpha(X,\ ):\mathbb{R^+}\to
\mathbb R^+_0$ is defined by
\begin{displaymath}
    \alpha(X,\varepsilon) = 1-\inf\{\mu(A_\varepsilon);\;A\subseteq X, A\text{
    is Borel}, \mu(A)\ge 1/2\}.
\end{displaymath}
In words, $\alpha(X,\varepsilon)$ measures how much space remains in $X$ when
one half of $X$ is inflated by $\varepsilon$.

Let $\mathcal X = \{(X_n,\rho_n,\mu_n);\; n = 1$, $2$, $\dots\}$ be a family
of metric probability spaces. Then $\mathcal X$ is called a \emph{normal Levy
family} with constants $c_1$, $c_2$ if for every $\varepsilon>0$ and for
every $n$ we have $\alpha(X_n,\varepsilon)\le c_1e^{-c_2\varepsilon^2n}$.

Let $\rho_n$ be the (normalized Hamming) metric on $S_n$ defined by
\begin{displaymath}
    \rho_n(\pi,\sigma)=\frac{1}{n}|\{i;\;\pi(i)\ne\sigma(i)\}|,
\end{displaymath}
and let $\mu_n$ be the (normalized counting) measure on $S_n$ defined by
\begin{displaymath}
    \mu_n(\pi) = \frac{1}{n!}.
\end{displaymath}
Then $\{(S_n,\rho_n,\mu_n)\}$ is a normal Levy family with constants $c_1=2$,
$c_2=1/64$, according to \cite[Sec.\ 6.4]{MS}.

Although the defining condition for normal Levy families only restricts the
interplay of the measure and the metric in $(X_n,\rho_n,\mu_n)$, one can say
a lot about the behavior of reasonable functions $f_n:X_n\to \mathbb R$. We
will assume here that $f_n$ is Lipschitz with constant $1$ (i.e.,
$|f_n(x)-f_n(y)|\le \rho_n(x,y)$ for every $x$, $y\in X_n$), but a more
general requirement would do (cf.\ \cite{MS}).

So, assume that $f:(X,\rho,\mu)\to \mathbb R$ is Lipschitz with constant $1$.
Denote by $M_f$ the median value of $f$ on $X$, and let $A=\{x\in X;\;f(x)\le
M_f\}$, $B=\{x\in X;\;f(x)\ge M_f\}$. Then, by definition, $\mu(A)\ge 1/2$,
$\mu(B)\ge 1/2$, and $\mu(\{x\in X;\;|f(x)-M_f|\le
\varepsilon\}|\ge\mu(A_\varepsilon \cap B_\varepsilon)\ge
1-2\alpha(X,\varepsilon)$. When $X=X_n$ is a member of a normal Levy family,
we thus obtain
\begin{displaymath}
    \mu(\{x\in X_n;\;|f(x)-M_f|\le \varepsilon\})\ge
    1-2c_1e^{-c_2\varepsilon^2n}.
\end{displaymath}
When $X_n=S_n$ is equipped with the above metric and measure, we get
\begin{displaymath}
    \mu(\{x\in X_n;\;|f(x)-M_f|\le \varepsilon\})\ge
    1-4e^{-\varepsilon^2n/64}.
\end{displaymath}
This inequality explains why the values of $f$ on $S_n$ are packed near the
median. Moreover, with such a spike in the distribution, \emph{the median
will be close to the average value of $f$}.

We are about to clinch the argument with the following observation:

\begin{proposition} Let $(S_n,\rho_n,\mu_n)$ be as above. Then all functions
$f_n:S_n\to \mathbb R$ defined by $f_n(\pi)=\disp{\pi}/n$ are Lipschitz with
constant $1$.
\end{proposition}
\begin{proof}
Let $\pi$, $\sigma$ be two permutations in $S_n$. Then
\begin{align*}
    \frac{1}{n}|\disp{\pi}-\disp{\sigma}|&=\frac{1}{n^2}\left|
            \sum_{i=1}^n|i-\pi(i)|-\sum_{i=1}^n|i-\sigma(i)|\right|\\
       &\le\frac{1}{n^2}\left|\sum_{i=1}^n|i-\pi(i)-i+\sigma(i)|\right|
            =\frac{1}{n^2}\sum_{i=1}^n|\pi(i)-\sigma(i)|\\
       &\le\frac{1}{n^2}\cdot n\cdot
            |\{i;\pi(i)\ne\sigma(i)\}| = \rho_n(\pi,\sigma),
\end{align*}
and we are through.
\end{proof}

%% PRESCRIBED DISPLACEMENT

\subsection{Prescribed displacement}

\noindent Since $S_n$ is finite, the values of $\disp{\pi}/n$ for a fixed $n$
cannot cover the interval $[0,1/2]$. However, we can get arbitrarily close to
any value in $[0,1/2]$ if we allow $n$ to be sufficiently large; as we are
going to show.

The idea is to leave $\pi$ identical on a certain proportion of $N$ and
displace the remaining points as much as possible.

\begin{proposition} Let $d$ be such that $0\le d\le 1/2$. Then there is a sequence of
permutations $\pi_n\in S_n$ such that $\lim_{n\to\infty}\disp{\pi_n}/n = d$.
\end{proposition}
\begin{proof}
Let $\delta = \sqrt{2d}$, and let $u_n=\lceil \delta n/2\rceil$. Define
$\pi_n\in S_n$ as follows:
\begin{displaymath}
    \pi(i)=\left\{\begin{array}{ll}
        i+u_n,&1\le i\le u_n,\\
        i-u_n,&u_n+1\le i\le 2u_n,\\
        i,&i>2u_n.
    \end{array}\right.
\end{displaymath}
Then $\disp{\pi_n}/n = 2u_n^2/n^2 = 2\lceil \delta n/2\rceil^2/n^2$. Since both
$2(\delta n/2)^2/n^2$ and $2(\delta n/2 + 1)^2/n^2$ tend to $\delta^2/2 = d$
when $n$ approaches infinity, we are done by the Squeeze theorem.
\end{proof}

\section{Stretching with additive formula}

In this section, we answer the following question: \emph{For which permutations
$\pi\in S_n$ is $\stre{+}{\mathcal B}{\pi}$ maximal, where $\mathcal
B=\{\{i,i+1\};\;1\le i<n\}$}? Note that with this choice of $\mathcal B$ we
have
\begin{displaymath}
    \stre{+}{\mathcal B}{\pi} =
    \frac{|\pi(1)-\pi(2)|+|\pi(2)-\pi(3)|+\cdots +|\pi(n-1)-\pi(n)|}{n-1}.
\end{displaymath}

For two subsets $A$, $B$ of $N$, we say that $\pi\in S_n$ \emph{oscillates}
between $A$ and $B$ if for every $1\le i<n$ we have either $\pi(i)\in A$,
$\pi(i+1)\in B$, or $\pi(i)\in B$, $\pi(i+1)\in A$.

\begin{theorem}\label{Th:StreA}
The maximum value of $\stre{+}{\mathcal B}{\pi}$ over all $\pi\in S_n$ is
\begin{align*}
    &(2m^2-1)/(2m-1)&\text{when $n=2m$, and}\\
    &(2m^2+2m-1)/(2m)&\text{when $n=2m+1$.}
\end{align*}

When $n=2m$, the maximum is attained by precisely those permutations $\pi$ that
oscillate between $\{1,\dots,m\}$, $\{m+1,\dots,n\}$ and satisfy
$(\pi(1),\pi(n))\in\{(m,m+1)$, $(m+1,m)\}$.

When $n=2m+1$, the maximum is attained precisely by those permutations $\pi$
that oscillate between $\{1,\dots,m\}$, $\{m+1,\dots,n\}$ and satisfy
$(\pi(1),\pi(n))\in\{(m+1,m+2)$, $(m+2,m+1)\}$, and by those that oscillate
between $\{1,\dots,m+1\}$, $\{m+2,\dots,n\}$ and satisfy
$(\pi(1),\pi(n))\in\{(m,m+1)$, $(m+1,m)\}$.
\end{theorem}
\begin{proof}
Let $n=2m$. Consider the sum $|\pi(1)-\pi(2)|+\cdots +|\pi(n-1)-\pi(n)|$. It
consists of $2n-2$ integers from $N$, $n-1$ with positive and $n-1$ with
negative signs. Now, if we are to maximize the sum of $2n-2$ integers out of
$1$, $1$, $\dots$, $n$, $n$ with $n-1$ integers having negative sign, we must
choose
\begin{equation}\label{Eq:Signs}
    -1-1-\cdots -(m-1)-(m-1)-m+(m+1)+(m+2)+(m+2)+\cdots +  n+n,
\end{equation}
which equals $2m^2-1$.

Is there a permutation $\pi$ such that
$|\pi(1)-\pi(2)|+\cdots+|\pi(n-1)-\pi(n)|=2m^2-1$? The fact that $m$, $m+1$
appear just once in \eqref{Eq:Signs} means that $\pi(1)=m$ and $\pi(n)=m+1$, or
vice versa. Moreover, the distribution of signs implies that $\pi$ must
oscillate between $\{1,\dots,m\}$ and $\{m+1,\dots,n\}$. Any such permutation
will do.

When $n=2m+1$, we proceed similarly. The two maximal sums analogous to
\eqref{Eq:Signs} are
\begin{displaymath}
    -1-1-\cdots -(m-1)-(m-1)-m-(m+1)+(m+2)+(m+2)+\cdots+n+n,
\end{displaymath}
and
\begin{displaymath}
    -1-1-\cdots -m-m+(m+1)+(m+2)+(m+3)+(m+3)+\cdots+n+n,
\end{displaymath}
since deleting both occurrences of $m+1$ would not correspond to any
permutation.
\end{proof}

% MULTIPLICATIVE STRETCHING

\section{Stretching with multiplicative formula}

We answer the following question: \emph{For which permutations $\pi\in S_n$
is $\stre{*}{\mathcal B}{\pi}$ maximal, where $\mathcal B=\{\{i,i+1\};\;1\le
i<n\}$?} Note that with this choice of $\mathcal B$ we have
\begin{displaymath}
    \stre{*}{\mathcal B}{\pi} =
    \left(\prod_{i=1}^{n-1}|\pi(i)-\pi(i+1)|\right)^{1/(n-1)}.
\end{displaymath}

\subsection{Maximizing products of $n$ integers with a given sum}

We obviously have:

\begin{lemma}\label{Lm:3Reals}
Let $x\le y$ be positive integers. Then $(x-1)(y+1)<xy$.
\end{lemma}

For positive integers $n\le s$, let
\begin{displaymath}
    D_{n,s}=\{(x_1,\dots,x_n);\;
        x_i\in\mathbb Z,\, x_i>0,\, x_1+\cdots + x_n=s\},
\end{displaymath}
and
\begin{displaymath}
    M_{n,s} = \max\{x_1\cdots x_n;\;(x_1,\dots,x_n)\in D_{n,s}\}.
\end{displaymath}
The following result is certainly well known. We offer a short proof:

\begin{theorem}\label{Th:MaxP}
Let $n\le s$ be positive integers, $a=s/n$. Then
\begin{displaymath}
    M_{n,s}=\lfloor a\rfloor^m\cdot \lceil a\rceil^{n-m},
\end{displaymath}
where $m = n\lceil a\rceil-s$. Moreover, $M_{n,s}<M_{n,s+1}$.
\end{theorem}
\begin{proof}
Let $\overrightarrow{x}=(x_1,\dots,x_n)$ be the unique point in $D$ such that
$x_1\le \cdots \le x_n$ and $x_n-x_1\le 1$. It is easy to see that
$x_1=\cdots = x_m=\lfloor a\rfloor$, $x_{m+1}=\cdots =x_n=\lceil a\rceil$,
where $m=n\lceil a\rceil -s$.

Let $\overrightarrow{y}=(y_1,\dots,y_n)\in D$ be such that $y_i\le y_{i+1}$
and $\overrightarrow{y}\ne \overrightarrow{x}$. Let $d_i=y_i-x_i$ and note
that $d_1<0$, $d_n>0$, $d_1+\cdots + d_n=0$. Assume for a while that $d_i>0$
and $d_j<0$ for some $i<j$. Then $x_i<y_i\le y_j<x_j$ shows that $x_i$, $x_j$
differ by more than $1$, which is impossible. Hence there is $k$ such that
$d_i\le 0$ for every $i\le k$, and $d_i\ge 0$ for every $i>k$.

The integers $d_i$ count how many times do we have to add or subtract $1$ to
obtain $y_i$ from $x_i$. Since $d_1+\cdots + d_n=0$, we can reach
$\overrightarrow{y}$ from $\overrightarrow{x}$ by repeatedly decreasing one
coordinate by $1$ and increasing other coordinate by $1$ at the same time.
Moreover, we have just shown that we can do this in such a way that only the
first $k$ coordinates will possibly decrease, and only the remaining $n-k$
coordinates will possibly increase. Since $x_k\le x_{k+1}$, Lemma
\ref{Lm:3Reals} implies that the product will diminish with every step.

It remains to show that $M_{n,s}<M_{n,s+1}$. When $(x_1$, $\dots$, $x_n)\in
D_{n,s}$ then $(x_1+1$, $x_2$, $\dots$, $x_n)\in D_{n,s+1}$, and, clearly,
$x_1\cdots x_n<(x_1+1)x_2\cdots x_n$.
\end{proof}

\subsection{The even case}

Let $n=2m$. Theorem \ref{Th:StreA} shows that $(n-1)\stre{+}{\mathcal
B}{\pi}\le 2m^2-1$, and that the equality holds if and only if $\pi$
oscillates between $\{1,\dots,m\}$, $\{m+1,\dots,n\}$ and
$(\pi(1),\pi(n))\in\{(m,m+1)$, $(m+1,m)\}$. By Theorem \ref{Th:MaxP}, the
product of $2m-1$ positive integers with sum $2m^2-1$ is maximized by $m\cdot
m + (m-1)(m+1)$.

\begin{lemma}\label{Lm:EvenP1}
Let $n=2m$. Let $\pi\in S_n$ be a permutation oscillating between
$\{1,\dots,m\}$, $\{m+1,\dots,n\}$ such that $\pi(1)=m$, $\pi(n)=m+1$ and such
that $|\pi(i)-\pi(i+1)|\in\{m,m+1\}$ for every $1\le i<n$. Then $\pi$ is
uniquely determined, namely: $\pi(2i)=n-i+1$, $\pi(2i-1)=m-i+1$.
\end{lemma}
\begin{proof}
We must have $\pi(2)=2m$. Then $\pi(3)=m-1$ since $\pi(1)=m$, etc.
\end{proof}

Dually:

\begin{lemma}\label{Lm:EvenP2}
Let $n=2m$. Let $\pi\in S_n$ be a permutation oscillating between
$\{1,\dots,m\}$, $\{m+1,\dots,n\}$ such that $\pi(1)=m+1$, $\pi(n)=m$ and such
that $|\pi(i)-\pi(i+1)|\in\{m,m+1\}$ for every $1\le i<n$. Then $\pi$ is
uniquely determined, namely: $\pi(2i)=i$, $\pi(2i-1)=m+i$.
\end{lemma}

\begin{theorem}\label{Th:StreMEven}
Let $n=2m$. Then the maximum of $\mathrm{s}^*_\mathcal B$ over all permutations
of $S_n$ is $(m^m(m+1)^{m-1})^{1/(2m-1)}$, and it is attained precisely by the
two permutations of Lemmas $\ref{Lm:EvenP1}$, $\ref{Lm:EvenP2}$.
\end{theorem}
\begin{proof} Let $\pi\in S_n$. Let $x_i=|\pi(i)-\pi(i+1)|$, $s=x_1+\cdots+
x_{2m-1}$. Then $s\le 2m^2-1$ by Theorem \ref{Th:StreA}. If $s<2m^2-1$ then
$\stre{*}{\mathcal B}{\pi}^{n-1}\le M_{n,2m^2-2} < M_{n,2m^2-1}$ by Theorem
\ref{Th:MaxP}. If $s=2m^2-1$, we have $\stre{*}{\mathcal B}{\pi}^{n-1}\le
M_{n,s}=m^m\cdot (m+1)^{m-1}$, and the equality holds only for the two
permutations of Lemma \ref{Lm:EvenP1}, \ref{Lm:EvenP2}.
\end{proof}

% LOCAL IMPROVEMENTS

\subsection{Local improvements}

When $n=2m+1$, we are going to see that the maximum of
$(\mathrm{s}^*_{\mathcal B})^{n-1}$ is $M=m^m(m+1)(m+2)^{m-1}$, which is far
less than $M_{2m,2m^2+2m-1}$ (cf. Theorems \ref{Th:StreA} and \ref{Th:MaxP}).
In fact, it can happen that $M<M_{2m,s}$ even if $s<2m^2+2m-1$. A more
detailed understanding of permutations $\pi$ with maximal
$\mathrm{s}^*_{\mathcal B}(\pi)$ is therefore needed.

There is a one-to-one correspondence between the permutations of $S_n$ and the
$n$-cycles of $S_n$ with designated beginning. To see this, identity $\pi\in
S_n$ with the $n$-cycle $\rho$ defined by $\rho(\pi(i))=\pi(i+1)$ if $i<n$,
$\rho(\pi(n))=\pi(1)$, and designate $\pi(1)$ as the beginning of $\rho$.
Therefore, finding the maximum of $\mathrm{s}^*_\mathcal B$ on $S_n$ is
equivalent to finding the maximum of $\mathrm{s}^*$ over all $n$-cycles $\rho$
in $S_n$, where
\begin{displaymath}
    \sstr{\rho}=\max\left\{\prod_{i\ne j}|i-\rho(i)|;\;1\le j\le n\right\}.
\end{displaymath}

In this subsection we show that a number of conditions on $\rho$ must hold,
should $\sstr{\rho}$ be maximal.

The following terminology will allow us to communicate more efficiently. We say
that two jumps $a\mapsto\rho(a)$, $b\mapsto\rho(b)$ of a cycle $\rho$
\emph{have distinct endpoints} if $|\{a,\rho(a),b,\rho(b)\}|=4$. The two jumps
are \emph{disjoint} if the intervals $[a,\rho(a)]$, $[b,\rho(b)]$ do not
intersect. The jump $a\mapsto \rho(a)$ \emph{skips over} the jump
$b\mapsto\rho(b)$ if $[b,\rho(b)]\subseteq [a,\rho(a)]$. (Note that a jump
skips over itself.) The jump $a\mapsto \rho(a)$ \emph{bridges}
$b\mapsto\rho(b)$ if it skips over it and the two jumps have distinct
endpoints. Two jumps \emph{intersect nontrivially} if they are not disjoint,
one does not skip over the other, and they have distinct endpoints. A jump
$a\mapsto\rho(a)$ is \emph{short} if $|a-\rho(a)|\le|b-\rho(b)|$ for all $b$.
All other jumps are called \emph{long}. Finally, the jumps have the \emph{same
direction} if $(a-\rho(a))(b-\rho(b))>0$, otherwise they have \emph{opposite
direction}.

Given a cycle $\rho$ and two jumps $i\mapsto\rho(i)$, $j\mapsto\rho(j)$ with
distinct endpoints, let $\rho_{i,j}$ denote the cycle depicted in Figure
\ref{Fg:Cross}.

%FIGURE
\setlength{\unitlength}{1mm}
\begin{figure}[ht]\begin{center}\input{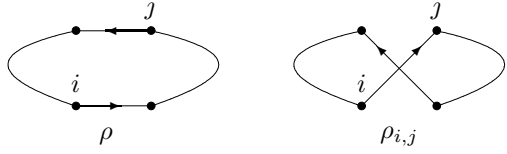}\end{center}
\caption{The cycles $\rho$ and $\rho_{i,j}$.}\label{Fg:Cross}
\end{figure}

\begin{lemma}\label{Lm:Aux1}
Let $\rho\in S_n$ be an $n$-cycle. Let $i\mapsto\rho(i)$, $j\mapsto\rho(j)$ be
jumps with distinct endpoints such that $i\mapsto\rho(i)$ is a short jump and
$|i-j|>|j-\rho(j)|$. Then $\sstr{\rho_{i,j}}>\sstr{\rho}$.
\end{lemma}
\begin{proof}
Since $i\mapsto\rho(i)$ is short, $\sstr{\rho}=\prod_{k\ne i}|k-\rho(k)|$. Now,
$\sstr{\rho_{i,j}}\ge |i-j|\prod_{k\ne i,\,k\ne j}|k-\rho(k)| > \prod_{k\ne
i}|k-\rho(k)|$.
\end{proof}

\begin{lemma}\label{Lm:LI}
Let $\rho\in S_n$ be an $n$-cycle such that one of the following conditions
holds:
\begin{enumerate}
\item[(i)] there are disjoint jumps in the same direction,

\item[(ii)] a short jump nontrivially intersects a jump in opposite
direction,

\item[(iii)] a short jump is disjoint from a jump in opposite direction,

\item[(iv)] there are disjoint jumps in opposite direction $($generalizing
\emph{(iii)}$)$,

\item[(v)] a jump bridges a long jump in opposite direction.
\end{enumerate}
Then there is an $n$-cycle $\sigma\in S_n$ such that
$\sstr{\sigma}>\sstr{\rho}$.
\end{lemma}
\begin{proof}
In case (i), write $a<\rho(a)<b<\rho(b)$ without loss of generality, and let
$\sigma=\rho_{a,b}$. Note that the two old jumps $a\mapsto\rho(a)$,
$b\mapsto\rho(b)$ have been replaced by two longer jumps $a\mapsto b$,
$\rho(a)\mapsto \rho(b)$, respectively.

In case (ii), let $a\mapsto\rho(a)$ be a short jump, and let $b$ be such that
$a<\rho(b)<\rho(a)<b$. Let $\sigma=\rho_{a,b}$ and note that the new jump
$a\mapsto b$ is longer that the old jump $b\mapsto\rho(b)$. We are done by
Lemma \ref{Lm:Aux1}.

In case (iii), let $a\mapsto\rho(a)$ be a short jump and $a<\rho(a)<\rho(b)<b$.
Let $\sigma=\rho_{a,b}$. The new jump $a\mapsto b$ is then longer than the old
jump $b\mapsto\rho(b)$, and we are again done by Lemma \ref{Lm:Aux1}.

In case (iv), we can assume that none of the two jumps $a\mapsto \rho(a)$,
$b\mapsto\rho(b)$ in question is short, else (iii) applies. Let
$c\mapsto\rho(c)$ be a short jump.  We can assume that $c\mapsto\rho(c)$ is not
disjoint from $a\mapsto\rho(a)$ nor $b\mapsto\rho(b)$, otherwise either (i) or
(iii) applies.  Without loss of generality, assume
$\max\{a,\rho(a)\}<\max\{b,\rho(b)\}$. Since the two jumps are in opposite
directions, $c\mapsto\rho(c)$ cannot intersect both jumps trivially.  Again
without loss of generality, assume $c\mapsto\rho(c)$ intersects
$a\mapsto\rho(a)$ nontrivially. If $a\mapsto \rho(a)$, $c\mapsto\rho(c)$ are in
opposite direction, then (ii) applies. So suppose that they are in the same
direction. Then $c\mapsto\rho(c)$ and $b\mapsto\rho(b)$ are in opposite
direction, and we can assume that they intersect trivially, else (ii) applies.
But that is impossible.

In case (v), let $\rho(b)<a<\rho(a)<b$ and $\sigma=\rho_{a,b}$. Let $x$, $y$,
$z$ be the lengths $a-\rho(b)$, $\rho(a)-a$ and $b-\rho(a)$, respectively. Then
we have lost the factor $(x+y+z)y = xy+y^2+yz$ and gained the factor
$(x+y)(y+z)=xy+xz+y^2+yz$ while comparing $\sstr{\rho}$ to $\sstr{\sigma}$.
Hence $\sstr{\sigma}>\sstr{\rho}$.
\end{proof}

\subsection{Short jumps}

We say that a jump $a\mapsto\rho(a)$ is \emph{right} if $a<\rho(a)$, else it is
\emph{left}.

\begin{proposition}\label{Pr:Short}
Let $\rho\in S_n$ be an $n$-cycle with maximal $\sstr{\rho}$. Assume that
$\rho$ has a short jump $c\mapsto c+t$, $t>0$. Then one of the following
scenarios holds:
\begin{enumerate}
\item[(i)] $t=1$, all jumps skip over $c\mapsto c+1$, $n=2m$, $c=m$, there
are $m$ left and $m$ right jumps in $\rho$,

\item[(ii)] $t=1$, the only jump not skipping $c\mapsto c+1$ is the right
jump following it, $n=2m+1$, $c=m$, there are $m+1$ right and $m$ left jumps
in $\rho$,

\item[(iii)] $t=1$, the only jump not skipping $c\mapsto c+1$ is the right
jump preceding it, $n=2m+1$, $c=m+1$, there are $m+1$ right and $m$ left
jumps in $\rho$,

\item[(iv)] $t=2$, precisely two jumps do not skip over $c\mapsto c+2$ and
these jumps are right, $n=2m+1$, $c=m$, there are $m+1$ right and $m$ left
jumps in $\rho$.
\end{enumerate}
\end{proposition}
\begin{proof}
If there is $d$ such that $c<d<c+t$, consider $a$ such that $d=\rho(a)$. By
Lemma \ref{Lm:LI}(ii), $a<c$. Similarly, $\rho(c)<\rho(d)$. The three jumps
$c\mapsto \rho(c)$, $a\mapsto\rho(a)$, $\rho(a)\mapsto\rho(\rho(a))=\rho(d)$
are thus all right.

If $\rho(c)-c>2$, there are $c<d<e<\rho(c)$. As above, there are jumps
$a\mapsto d\mapsto \rho(d)$, $b\mapsto e\mapsto\rho(e)$, all right. But then
Lemma \ref{Lm:LI}(i) applies to $a\mapsto\rho(a)$ and $e\mapsto\rho(e)$, a
contradiction. Hence $t=\rho(c)-c\le 2$.

Assume $\rho(c)-c=2$ and let $a\mapsto\rho(a)=c+1\mapsto\rho(c+1)$ be the two
right jumps found above. Let $b\mapsto\rho(b)$ be a right jump different from
$a\mapsto c+1$, $c+1\mapsto\rho(c+1)$, $c\mapsto c+2$. Then $b<c$, else
$a\mapsto c+1$, $b\mapsto\rho(b)$ are disjoint and Lemma \ref{Lm:LI}(i)
applies. If $\rho(b)\le c$, the jump $b\mapsto\rho(b)$ is disjoint from
$\rho(a)\mapsto\rho(\rho(a))$, a contradiction with Lemma \ref{Lm:LI}(i). If
$\rho(b)>c$, we must have $\rho(b)>\rho(c)$, and so $b\mapsto\rho(b)$ skips
over $c\mapsto\rho(c)$. Now let $b\mapsto\rho(b)$ be any left jump. If $b<c+2$
then, in fact, $b<c$, thus $b\mapsto\rho(b)$ and $c\mapsto c+2$ are disjoint, a
contradiction by Lemma \ref{Lm:LI}(iii). Thus $b\ge c+2$. If $\rho(b)>c+1$ then
$b\mapsto\rho(b)$, $a\mapsto\rho(a)$ are disjoint and Lemma \ref{Lm:LI}(iv)
applies. If $\rho(b)\le c+1$, we must have $\rho(b)\le c$, and
$b\mapsto\rho(b)$ skips over $c\mapsto\rho(c)$. The rest of (iv) is easy.

The case $\rho(c)-c=1$ can be analyzed similarly, with help of Lemma
\ref{Lm:LI}.
\end{proof}

In view of Theorem \ref{Th:StreMEven}, we are only interested in scenarios
(ii), (iii) and (iv) of Proposition \ref{Pr:Short}.

\subsection{Long jumps}

The following Lemma follows immediately from Lemma \ref{Lm:LI}(iv), (v):

\begin{lemma}\label{Lm:Under}
Let $\rho$ be an $n$-cycle with maximal $\sstr{\rho}$. Let $a\mapsto\rho(a)$,
$b\mapsto\rho(b)$ be two long jumps of opposite directions. Then at least one
of the endpoints of $b\mapsto\rho(b)$ is in the interval $[a,\rho(a)]$.
\end{lemma}

\begin{proposition}\label{Pr:Lengths}
Let $\rho\in S_n$ be an $n$-cycle with maximal $\sstr{\rho}$ and with a short
cycle $c\mapsto c+t$, $t>0$, where $n=2m+1$. Then every long jump of $\rho$ is
of length $m$, $m+1$ or $m+2$.
\end{proposition}
\begin{proof}
Let $k\mapsto k+t$, $0<t<m$, be a long right jump of $\rho$. By Proposition
\ref{Pr:Short}, $m+1$ is the unique point at which $2$ right jumps are
consecutive, and, moreover, $m+1\in[k,k+t]$. By the same Proposition, there are
$m$ left jumps, no two consecutive. By Lemma \ref{Lm:Under}, each of these left
jumps has an endpoint in $[k,k+t]$. Then there are not enough points in
$[k,k+t]$ for $m$ nonconsecutive left jumps to start or end at.

Let $k\mapsto k-t$, $0<t<m$, be a left jump of $\rho$. By Proposition
\ref{Pr:Short} and Lemma \ref{Lm:Under}, there are $m$ long right jumps and
each of them has an endpoint in $[k-t,k]$. In scenario (ii) of Proposition
\ref{Pr:Short}, $m\in [k-t,k]$, no long right jump starts or ends at $m+1$, and
no two long right jumps are consecutive. In scenario (iii), $m+2\in[k-t,k]$, no
long right jump starts or ends at $m$, and no two long right jumps are
consecutive. In scenario (iv), $m$, $m+2\in[k-t,k]$, no long right jump starts
or ends at $m$, $m+2$, and precisely two long right jumps are consecutive. In
any case, there are not enough points in $[k-t,k]$ to accommodate all long
right jumps.

Consider a jump $a\mapsto\rho(a)$ of length at least $m+3$. Then there are at
most $2m+1-(m+2)=m-1$ points outside of $(a,\rho(a))$. Assume that $a<\rho(a)$.
Then one of the $m$ left jumps, no two of which are consecutive, must have both
endpoints in $(a,\rho(a))$. Assume that $a>\rho(a)$. Note that no point outside
of $(a,\rho(a))$ can be both the starting and the terminating point of a right
jump (this is obvious for $a$, $\rho(a)$, and it is true for the remaining
points by Lemma \ref{Lm:LI}(iv)). Hence one of the $m+1$ long right jumps must
have both endpoints in $(a,\rho(a))$. In any case, we have reached a
contradiction by Lemma \ref{Lm:LI}(v).
\end{proof}

\begin{lemma}\label{Lm:ScenarioII} Let $\rho$ be as in scenario \emph{(ii)}
of Proposition $\ref{Pr:Short}$. Then every long jump is of length $m$, $m+1$,
or $m+2$, $\rho$ is uniquely determined, and $\sstr{\rho}=m^m\cdot (m+1)\cdot
(m+2)^{m-1}$. When $m$ is odd, we have
\begin{displaymath}
    \rho(i) = \left\{\begin{array}{ll}
        i+1,&i=m,\\
        i-(m+1),&i=m+2,\\
        i+m,&i\text{\ even},\,i<m+2,\\
        i+(m+2),&i\text{\ odd},\,i<m,\\
        i-m,&i\text{\ even},\,i>m+1,\\
        i-(m+2),&i\text{\ odd},\,i>m+2.
    \end{array}\right.
\end{displaymath}
When $m$ is even, we have
\begin{displaymath}
    \rho(i) = \left\{\begin{array}{ll}
        i+1,&i=m,\\
        i+(m+1),&i=1,\\
        i+m,&i\text{\ odd},\,1<i<m+2,\\
        i+(m+2),&i\text{\ even},\,i<m,\\
        i-m,&i\text{\ even},\,i>m+1,\\
        i-(m+2),&i\text{\ odd},\,i>m+2.
    \end{array}\right.
\end{displaymath}
\end{lemma}
\begin{proof}
We work out two examples, one for $m=3$ and one for $m=4$. It will then become
clear that the cycle $\rho$ is unique, that its structure is determined by the
parity of $m$, and that the formulae in the statement of the Lemma are correct.
We will build the cycle from the shortest jump $m\mapsto m+1$ by alternatively
extending it by one jump forward and one jump backwards.

Let $m=3$. By our assumption, $\rho(3)=4$. We now determine $\rho(4)$ (building
the cycle forward) and $\rho^{-1}(3)$ (building the cycle backwards). Since
$\rho(4)>4$ by the assumption, we must have $\rho(4)=7$ (else the jump is too
short). Then $\rho^{-1}(3)=6$, since $\rho^{-1}(3)=7$ would result in a short
cycle, and all other values yield a jump that is too short. We next determine
$\rho(7)$ and $\rho^{-1}(6)$. We must have $\rho(7)=2$, since $\rho(7)=1$ would
be too long. Then $\rho^{-1}(6)=1$ follows, avoiding a short cycle. Now we
obviously have $\rho(2)=5=\rho^{-1}(1)$.

Let $m=4$. By our assumption, $\rho(4)=5$. Proceeding as in the case $m=3$, we
have $\rho(5)=9$, $\rho^{-1}(4)=8$, $\rho(9)=3$, $\rho^{-1}(8)=2$, $\rho(3)=7$,
$\rho^{-1}(2)=6$, $\rho(7)=1$, and $\rho^{-1}(6)=1$.
\end{proof}

Similarly:

\begin{lemma}\label{Lm:ScenarioIII}
Let $\rho$ be as in scenario \emph{(iii)} of Proposition $\ref{Pr:Short}$. Then
every long jump is of length $m$, $m+1$, or $m+2$, $\rho$ is uniquely
determined, and $\sstr{\rho}=m^m\cdot (m+1)\cdot (m+2)^{m-1}$. The formulae for
$\rho$ are similar to those of Lemma $\ref{Lm:ScenarioII}$.
\end{lemma}

\begin{lemma}\label{Lm:LotOfM} Let $\rho$ be as in scenario \emph{(iv)} of Proposition
$\ref{Pr:Short}$. Then there are at least $m-1$ jumps of length $m$ in
$\rho$.
\end{lemma}
\begin{proof}
We use Proposition \ref{Pr:Lengths} without reference throughout this proof.

For $i\in \{1,\dots,m-1\}$, let $L(i)$ denote the length of the left jump
ending at $i$, and $R(i)$ the length of the right jump starting at $i$. Note
that we cannot have $L(i)=R(i)$, else a $2$-cycle arises. We claim that in at
most one case among $1$, $\dots$, $m-1$ both $L(i)$, $R(i)$ are bigger than
$m$, hence proving the lemma (since $m+1\mapsto 2m+1$ is also of length $m$).

For a contradiction, let $i<j$ be the two smallest integers in $\{1$,
$\dots$, $m-1\}$ such that $L(i)$, $R(i)$, $L(j)$, $R(j)>m$. Assume that
$L(i)=m+1$, $R(i)=m+2$. (The case $L(i)=m+2$, $R(i)=m+1$ is similar.) Let
$k=j-i$.

Assume $k=1$. Since $R(i+1)\ne m+1$, we have $L(i+1)=m+1$, $R(i+1)=m+2$. Since
$R(i+2)\ne m$ and $R(i+2)\ne m+1$, we have $R(i+2)=m+2$. Since $L(i+2)\ne m$,
we have $L(i+2)=m+1$. Continuing in this fashion, we arrive at $R(m-1)=m+2$,
contradicting $m+1\mapsto 2m+1$.

Assume $k=2$. Since $L(i+1)\ne m$, we have $R(i+1)=m$. If $L(i+1)=m+1$, we have
a $4$-cycle. Hence $L(i+1)=m+2$. Since $j=i+2$, $L(i+2)\ne m$. Also, $L(i+2)\ne
m+1$. Thus $L(i+2)=m+2$. But then the jump starting at $m+i+2$ is not of length
$m$, $m+1$, or $m+2$, a contradiction.

Assume $k=3$. Then $R(i+1)=m$, and thus $L(i+1)=m+2$ else we have a $4$-cycle.
Then $L(i+2)=m$, and thus $R(i+2)=m+2$ else we have a $6$-cycle. As $R(i+3)\ne
m$ and $R(i+3)\ne m+1$, we have $R(i+3)=m+2$. But then no jump can possibly end
at $m+i+3$, a contradiction.

This pattern continues for larger $k$.
\end{proof}

\subsection{The odd case}

\begin{theorem}\label{Th:StreMOdd} Let $n=2m+1>1$. Then the maximum of
$\mathrm{s}^*_{\mathcal B}$ over all permutations of $S_n$ is $(m^m\cdot
(m+1)\cdot(m+2)^{m-1})^{1/{n-1}}$, and it is attained precisely by the two
permutations of Lemmas $\ref{Lm:ScenarioII}$ and $\ref{Lm:ScenarioIII}$, and by
their mirror images.
\end{theorem}
\begin{proof}
Let $\rho$ be a permutation obtained in scenario (iv) of Proposition
\ref{Pr:Short}. Its $m$ left jumps start in positions $m+2$, $\dots$, $2m+1$,
and its $m$ long right jumps start in positions $1$, $\dots$, $m-1$, $m+1$. It
is then easy to see that the sum of the lengths of the $2m$ long jumps of
$\rho$ is $2m^2+2m-2$. By Proposition \ref{Pr:Lengths}, each long jump is of
length $m$, $m+1$ or $m+2$, and by Lemma \ref{Lm:LotOfM} there are at least
$m-1$ jumps of length $m$. If $x_1$, $\dots$, $x_{2m}$ are positive integers
such that $m\le x_i\le m+2$, $x_1+\cdots+x_{2m}=2m^2+2m-2$ and such that at
least $m-1$ of them are equal to $m$, then Theorem \ref{Th:MaxP} implies that
the product $x_1\cdots x_{2m}$ cannot exceed $m^{m-1}(m+1)^4(m+2)^{m-3}$.
However, $m^{m-1}(m+1)^4(m+2)^{m-3}$ is less than $m^m(m+1)(m+2)^{m-1}$ if and
only if $(m+1)^3$ is less than $m(m+2)^2$, which is true for every positive
$m$. We are done by Lemmas \ref{Lm:ScenarioII}, \ref{Lm:ScenarioIII} and by
their mirrored versions.
\end{proof}

\section{Acknowledgement}

We thank Alvaro Arias for telling us about the connection between probability
measure spaces and permutations, and to George Edwards for providing us with
papers on interleaver designs.

An early version of this paper dealt only with the displacement of
permutations, and mentioned the problem of maximizing $\mathrm{s}^*_{\mathcal
B}$ as an open question. An anonymous referee responded that he/she can answer
the question, that the maximum is as in Theorems \ref{Th:StreMEven} and
\ref{Th:StreMOdd}, and that he/she has ``an elementary but not very short
proof.'' The referee did not reveal the proof and did not indicate how many and
which permutations attain the maximum, but we are indebted to him/her for
pointing us in the right direction.

\bibliographystyle{plain}

\end{document}